\documentclass[]{amsart}

 \usepackage[]{amsmath, amsthm, amsfonts,amssymb}
 \usepackage[all]{xy}

 \newcommand {\C} {{\mathbb C}}
 
 \newcommand {\Z} {{\mathbb Z}}
 \newcommand {\Q} {{\mathbb Q}}

\newcommand {\OO} {{\mathcal O}}
\newcommand{\dt} {{\bullet}}

\newcommand {\cP}{\mathcal{P}}

 \newtheorem{thm}[subsection]{Theorem}
 \newtheorem{cor}[subsection]{Corollary}
 \newtheorem{lemma}[subsection]{Lemma}
 \newtheorem{prop}[subsection]{Proposition}

\begin{document}

 \title[The prime to $p$ fundamental group]{Restrictions on the prime to $p$ fundamental group of a smooth projective variety}
 
  \author{Donu Arapura}
\thanks{Partially supported by NSF}
 \address{Department of Mathematics\\
   Purdue University\\
   West Lafayette, IN 47907\\
   U.S.A.}

%\classification{14H30}
%\keywords{fundamental group, projective variety}

 \begin{abstract}
   The goal of this paper is to obtain restrictions on the prime to p quotient of the étale fundamental group of a smooth projective variety in characteristic $p\ge 0$. The results are analogues some theorems in the study of Kähler groups. Our first main result is that such groups are indecomposable under coproduct. The second gives a classification of the pro-$\ell$ parts of one relator groups in this class.
 \end{abstract}

 \maketitle

Our goal is obtain some analogues, over fields of characteristic $p\ge
0$, of a few of the known  restrictions on the class of   K\"ahler groups; we
recall that a group is K\"ahler if it can be realized as the
fundamental group of a compact K\"ahler manifold.
In this paper, we  replace the usual fundamental group by the maximal
prime to $p$ quotient of the \'etale fundamental group.
We focus on this group, because it behaves most like it
 topological namesake.
Let $\cP(p)$ denote the class of profinite groups that arise
as  prime to $p$ fundamental groups of  smooth
projective varieties defined over an algebraically closed field of
characteristic $p$. 
Our first main result implies, among other things, the  indecomposibility of  groups in $\cP(p)$ under
coproduct. This is an analogue of Gromov's theorem  \cite{gromov} in
the K\"ahler setting. For the  second main result, we show that if
$G\in \cP(p)$ is the completion of a one relator group, then for almost every 
$\ell$, the pro-$\ell$ quotient $G_\ell$ of $G$
is isomorphic to the  pro-$\ell$ fundamental group of a smooth
projective curve. This  is inspired by the recent  classification of
one relator K\"ahler groups of Biswas, Mahan \cite{bm} and Kotschick
\cite{kotschick}, although the argument here  is completely different.  We deduce from the hard
Lefschetz theorem that $G_\ell$ is a Demushkin group for
almost all $\ell$, then
the result follows from the classification of such groups.

 I would like to thank   Jakob Stix for bringing  Demuskin's work to my attention.
%\acknowledgements{ I would like to thank   Jakob Stix for bringing  Demuskin's work to my attention.}

\section{Preliminaries}

From the beginning, we fix an integer $p$ which is either a prime number or $0$. By a $p'$-group we will mean a finite group of
order prime to $p$ (or arbitrary when $p=0$), and by a pro-$p'$ group, we mean an inverse limit of
$p'$-groups.  The symbol $\ell$ will always stand for  a fixed or
variable prime different from $p$.
Given a pro-finite group $G$, let $G_{{p'}}$, respectively $G_\ell$,
denote the maximal pro-$p'$, respectively pro-$\ell$, quotient of
$G$. Given a discrete group $G$, we let
$$\hat G = \varprojlim_{G/N\text{ a $p'$-group}} G/N$$
denote the  pro-$p'$ completion. So that $\hat \Z= \prod_{\ell\not= p}\Z_\ell$.
Then $\hat G_\ell$ can be identified with the  pro-$\ell$ completion
of $G$.
Given a connected scheme $X$, let $\pi_1^{et}(X)$
denote Grothendieck's  \'etale fundamental group \cite{sga1}, where we
ignore the base point. This is the profinite group for which the category of
finite sets with continuous action is equivalent to the category of
\'etale covers of $X$.
Let us write $\pi_1^{{p'}}(X)$ and
$\pi_1^\ell(X)$ instead of $\pi_1^{et}(X)_{{p'}}$ and $\pi_1^{et}(X)_\ell$.
 Given an
algebraically closed  field $k$ of characteristic $p$, let $\cP(k)$
denote the class of pro-$p'$ groups which are isomorphic to
$\pi_1^{{p'}}(X)$, where $X$ is a smooth projective $k$-variety.
$\cP(\C)$ is the class of profinite completions of  topological
fundamental groups of complex smooth projective varieties.
Set $\cP(p)=\cP(\overline{\mathbb{F}}_p)$, where $\overline{\mathbb{F}}_p$
is the algebraic closure of the prime field of characteristic $p$
(so that  $\overline{\mathbb{F}}_0=\bar \Q$). There is no loss in
focusing on this case because of the following fact.

\begin{prop}
  If $k$ is an algebraically closed field of characteristic $p$,
  $\cP(k)= \cP(p)$.
\end{prop}

\begin{proof}
  Clearly $\cP(p)\subseteq \cP(k)$ because extension of scalars from
  $\overline{\mathbb{F}}_p$ to $k$ will not change the fundamental
  group \cite[exp X, cor 1.8]{sga1}. Suppose that $X$ is a smooth projective
  $k$-variety. It is defined over a finitely generated extension $K$
  of $\overline{\mathbb{F}}_p$, i.e. there exists a $K$-scheme $X_K$
  such that $X=X_K\times_{Spec\, K} Spec\, k$. Let $S$ be variety defined over
  $\overline{\mathbb{F}}_p$ with function field $K$. After shrinking
  $S$, if necessary, we can assume that there is a smooth projective
  morphism $\mathcal{X}\to S$ with geometric generic fibre $X_K$. Choose an
  $\overline{\mathbb{F}}_p$ rational point $y_0\in S$ and let $\eta$
  denote the geometric generic point. Then
  $\pi_1^{{p'}}(\mathcal{X}_{y_0})\cong \pi_1^{{p'}}(\mathcal{X}_{\eta})\cong\pi_1^{{p'}}(X)$ by \cite[exp X,
  cor 3.9]{sga1}.
\end{proof}

\begin{lemma}\label{lemma:fi}
  If $G\in \cP(p)$ and $H\subset G$ is open, then $H\in \cP(p)$.
\end{lemma}

\begin{proof}
   If $G=\pi_1^{{p'}}(X)$ then $H=\pi_1^{{p'}}(Y)$ for some \'etale cover  $Y\to X$.
\end{proof}

Given a finitely generated $\Z_\ell$ module $V$ 
with a continuous action by a profinite
group $G$, we define
$$H^i(G, V)
:=\varprojlim_n H^i(G,
V/\ell^nV)$$
$$H^i(G,V\otimes \Q_\ell)=H^i(G,V)\otimes \Q_\ell$$
This  naive definition will suffice  for our
purposes, although there  is one place  where we are better off with the
more subtle definition of  Jannsen \cite{jannsen}.  We summarize what we need about
this in the following lemma.

\begin{lemma}\label{lemma:HS}
  Suppose that $1\to K\to G\to H\to 1$ is an exact sequence of profinite groups
  with $G$ topologically finitely generated and $V$ a finitely
  generated $\Z_\ell$-module with continuous $H$ action.
 Then there is the usual $5$-term exact sequence of Hochschild-Serre
$$0\to H^1(H,V)\to H^1(G,V)\to H^0(H, H^1(K,V))\to H^2(H,V)\to H^2(G,V)$$
\end{lemma}

\begin{proof}
Following Jannsen, we define $H^i_{cont}(G,V)$ as
the $i$th derived functor of 
$$V\mapsto \varprojlim_n H^0(V/\ell^nV)$$
The Hochschild-Serre spectral sequence, and the resulting $5$ term
sequence, for $H^*_{cont}$ can be constructed in the usual way.
By \cite[(2.1)]{jannsen}, we have an exact sequence 
\begin{equation}\label{eq:lim1}
0\to {\varprojlim}^1 H^{i-1}(G, V/\ell^n V)\to H^i_{cont}(G,V)\to H^i(G,V)\to 0  
\end{equation}

When $i\le 2$, we claim that $H^{i-1}(G,V/\ell^n V)$
is finite. For $i=1$ this is clear  because $V$ is finitely generated. So we have to show this for $i=2$.
We can find an open normal subgroup $G_1$ which acts trivially on $V/\ell^nV$.
Then by Hochschild-Serre, we have an exact sequence
$$H^1(G/G_1,V/\ell^nV)\to H^1(G,V/\ell^nV)\to H^1(G_1, V/\ell^nV)$$
The finiteness of the middle group is a consequence of the finiteness of the outer groups.
The first group is finite, because both $G/G_1$ and $V/\ell^nV$ are. The
last group 
$H^{1}(G_1,V/\ell^n V) $ is isomorphic to $Hom_{cont}(G_1, V/\ell^n V)$.
 By assumption, $G$ contains a finitely generated 
dense subgroup $\Gamma$. The group $\Gamma\cap G_1$ is easily seen to be finitely generated and dense in $G_1$.
Therefore $Hom_{cont}(G_1, V/\ell^n V)$ finite and the claim is proved.

The   Mittag-Leffler condition holds for $H^{i-1}(G,V/\ell^n V)$  and $i\le 2$ by the previous claim.
Therefore  the $\varprojlim^1$ in \eqref{eq:lim1}
vanishes,  and so $H^i(G,V)\cong H^i_{cont}(G,V)$ for $i\le 2$. By the same argument, $H^i(H,V)\cong H^i_{cont}(H,V)$ .
So the $5$-term sequence for $H^*_{cont}$ can be
identified with the one given  in the statement of the lemma.
\end{proof}

\begin{lemma}\label{lemma:Gabl}
  If $G$ is a profinite group and $V$ an abelian pro-$\ell$ group, then
$$(G/[G,G])_\ell\cong G_\ell/[G_\ell,G_\ell]$$
$$Hom(G,V)\cong Hom(G_\ell,V)$$
where $Hom$ is the group of continuous homomorphisms.
\end{lemma}

\begin{proof}
It is enough to prove the first isomorphism, because the second is a
consequence of it.
Since $(G/[G,G])_\ell$ is an abelian pro-$\ell$ group, the homomorphism $G\to
(G/[G,G])_\ell$ factors through the abelianization of the maximal
pro-$\ell$ quotient $G_\ell/[G_\ell,G_\ell]$. So we have a
homomorphism $G_\ell/[G_\ell,G_\ell]\to (G/[G,G])_\ell$.
On the
other hand, the map $G\to G_\ell/[G_\ell,G_\ell]$ must factor through
the maximal pro-$\ell$ quotient of the abelianization
$(G/[G,G])_\ell$.
This gives the inverse.
\end{proof}

\begin{prop}\label{prop:cup}
Suppose that $X$ is a connected scheme of finite type over $k$.
  Let $G$ be a quotient of $\pi_1^{et}(X)$ by a closed
  normal subgroup, such that $G$ dominates
 $\pi_1^\ell(X)$. Given a finitely generated
  $\Z_\ell$-module $V$ with continuous  $G$-action, there exists
a homomorphism to $\ell$-adic cohomology
\begin{equation}
  \label{eq:grouptoetale}
H^i(G, V) \to H^i(X,V)
\end{equation}
This is compatible with the cup products
$$H^i(G, V)\otimes H^j(G, V')\to H^{i+j}(G, V\otimes V')$$
and 
$$H^i(X, V)\otimes H^{j}(X, V')\to H^{i+j}(X, V\otimes V')$$
 The map \eqref{eq:grouptoetale} is an isomorphism when $i\le 1$.
\end{prop}

\begin{proof}
We start by proving the analogous statements over
$\Lambda_n=\Z/\ell^n\Z$, and then take the limit.
We indicate two different constructions of the map; the first is
simpler, but second gives more, and so it is the one that we use. First of all,
both $H^i(G,-)$ and $H^i(X,-)$ are $\delta$-functors
  from the category of discrete $\Lambda_n[G]$-modules to abelian groups,
  with the first universal in the sense of \cite{groth}.
 By the connectedness assumption $H^0(G, V) \cong
  H^0(X,V)$. Thus we get a map
$$H^*(G, -) \to H^*(X,-)$$
of $\delta$-functors.  
Compatibility with cup products can be proved in principle by
dimension shifting and induction, but it
seems simpler to give an alternate interpretation. 
Suppose that $Y\to X$ is a Galois \'etale cover
with Galois group $H$ a quotient of $G$ by an open normal subgroup. 
Then we have an isomorphism of simplicial schemes
$$cosk(Y\to X)_\dt \cong (Y\times EH_\dt)/G$$ where  $cosk(Y\to X)_\dt$
is the simplicial scheme
$$\ldots Y\times_X Y\rightrightarrows Y$$
and $EH_\dt\to BH_\dt$
is a simplicial model for the universal $H$-bundle over the
classifying space  (cf \cite[\S 5.1, 6.1]{deligneH} or \cite[pp 99-100]{milne}). The projection $cosk(Y\to
X)_\dt \to EH_\dt/G=BH_\dt$ induces a map from the
bar complex $C^\dt(H, V)$ with coefficients in an $H$-module $V$ to the
{C}ech complex $\check{C}^\dt({Y\to X}, V)$. Thus we obtain maps 
$$H^*(G, V) \to H^*(H,V)\to H^*(X,V)$$
The compatibility with cup products now follows easily from the standard
simplicial formulas for them \cite[p 172]{milne}.

We have already seen that the map \eqref{eq:grouptoetale} is an isomorphism when $i=0$.
We next prove that it  is an isomorphism when $i=1$.  First, suppose
that $V$ is a finite $\Lambda_n$-module with trivial $G$-action.
Then  have an isomorphism
$$H^1(G, V)\cong Hom(G, V)$$
On the other hand, we have
$$H^1(X,V) \cong Hom(\pi_1^{et}(X),V)$$
because both  groups classify $V$-torsors
\cite[pp 121-123]{milne}.  By lemma \ref{lemma:Gabl}, we can also identify
$$Hom(\pi_1^{et}(X),V)\cong Hom(G,V)$$
Now suppose that $V$ is a nontrivial finite
$\Lambda_n[G]$-module. Let $\pi:Y\to X$ be an \'etale cover such that
$\pi^*V$ is trivial. We can assume that $\pi$ is  a Galois  cover, with
Galois group $H$ a quotient of $G$. Set $K=\pi_1^{et}(Y)$. Then $K$
acts trivially on $\pi^*V$. It follows that 
\begin{equation}
  \label{eq:grouptoetale2}
H^i(K,V)\cong H^i(Y,\pi^*V),\quad  i=0,1
\end{equation}
and this isomorphism is compatible with the $H$ action.
Then Hochschild-Serre gives
a commutative diagram 
$$
\xymatrix{
 0\ar[r] & H^1(H,  H^0(K,V))\ar[r]\ar[d]^{\cong} & H^1(G,V)\ar[r]\ar[d]^{f} & H^0(H, H^1(K,V))\ar[r]\ar[d]^{\cong} & H^2(H, H^0(K,V))\ar[d]^{\cong} \\ 
 0\ar[r] & H^1(H,H^0(Y,V))\ar[r] & H^1(X,V)\ar[r] & H^0(H,H^1(Y,V))\ar[r] & H^2(H,H^0(X,V))
}
$$
with exact rows. The maps labeled by $\cong$ are isomorphisms by
\eqref{eq:grouptoetale2}. Thus $f$ is an isomorphism by the $5$-lemma.

To summarize, we have canonical multiplicative homomorphisms
$$H^i(G, V) \to H^i(X,V)$$
for $\Lambda_n[G]$-modules,
which are isomorphisms for $i\le 1$.  The  proposition follows by taking the inverse limit over $n$.
\end{proof}

We will apply the last proposition in the two cases $G=\pi_1^{p'}(X)$
and $G=\pi_1^\ell(X)$. It is worth remarking that  when $V=\Q_\ell$, we have $H^1(\pi_1^{p'}(X),
V)\cong H^1(\pi_1^\ell(X),V)$, but there is no reason to expect this
for higher cohomology.

We have the following basic finiteness property. 

\begin{thm}[Raynaud ]
Any element of $P(p)$ is topologically finitely presented.
\end{thm}

\begin{proof}This follows from \cite[thm 2.3.1, rem 2.3.2]{raynaud}. \end{proof}

The analogous statement for K\"ahler groups is a well known consequence of the
finite triangulability of compact manifolds. We wish to point out that
topological finite presentability does not preclude some  fairly wild examples
such as $\prod_{\ell\not= p} \Z\ell\Z$. However, such examples cannot lie in
$\cP(p)$.

\begin{prop}\label{prop:even}
  If $G\in P(p)$, then $G/[G,G]$ is the product of a  finite abelian
  group with $\hat \Z^{b}=\prod_{\ell\not= p} \Z_\ell^{b}$
  where $b$ is an even integer. 
\end{prop}

\begin{proof}
  We can decompose $G/[G,G]=\prod_{\ell\not= p} \Z_\ell^{b_\ell}\times A_\ell$,
  where $A_\ell$ is a finite abelian $\ell$-group. We have to show
  that $b_\ell$ is constant and that $A_\ell=0$ for $\ell\gg 0$.
 The Kummer sequence \cite[p 66]{milne} gives an isomorphism
$$Hom(\prod \Z_\ell^{b_\ell}\times A_\ell, \Z_\ell)\cong
 T_\ell Pic(X)=T_\ell Pic^0(X)_{red}$$
where we identify $\Z_\ell \cong \Z_\ell(1)$.
For the last equality, we use the exact sequence
$$0\to Pic^0(X)\to Pic(X)\to NS(X)\to 0$$
and the fact that  the Neron-Severi group $NS(X)$  is finitely
generated.  Since $Pic^0(X)_{red}$ is an abelian variety, it follows
that $b_\ell = b= 2\dim Pic^0(X)_{red}$ (see for example \cite[thm 15.1]{milneA}).

Again by Kummer, we have an isomorphism
$$Hom(\prod \Z_\ell^{b}\times A_\ell, \Z/\ell \Z)\cong
 \ell\text{-torsion subgroup of }Pic(X)$$
Since $NS(X)$ is finitely generated, the $\ell$-torsion subgroups of $Pic(X)$ and $Pic^0(X)$
coincide for all $\ell\gg0$. The $\ell$-torsion subgroup of $Pic^0(X)$
is isomorphic to $(\Z/\ell\Z)^b$. Therefore for $\ell\gg 0$, we must have
$Hom(A_\ell, \Z/\ell\Z)=0$ which implies that $A_\ell=0$. 
\end{proof}

\section{Consequences of hard Lefschetz}

By far the simplest restriction on K\"ahler groups  is what we will refer to as the parity test: a
finitely generated $\Gamma$ cannot be K\"ahler  unless $rank(\Gamma/[\Gamma,\Gamma])$ is even.
This is a consequence of the Hodge decomposition.  Proposition
\ref{prop:even} gives an analogue in our setting. It is convenient to
record the relevant part of it as a corollary.

\begin{cor}\label{cor:even}
If $G\in \cP(p)$,  $rank \, G_\ell/[G_\ell,G_\ell]$ is a fixed even
integer for each $\ell\not=p$.
\end{cor}

The following fact, which refines the previous result, was first observed by Johnson and Rees \cite{jr} in
the K\"ahler group setting.

\begin{thm}\label{thm:even}
  Let $G\in \cP(p)$,  and let $H$ be a quotient of $G$ by a closed normal subgroup
  such that $H$   dominates $G_\ell$. Suppose that $\rho:H\to
  O_n(\Q_\ell)$ is an orthogonal representation such
  that $\rho(H)$ is finite, and let  $V$ be the corresponding
  $H$-module with quadratic form $q:V\otimes V\to \Q_\ell$. Then there exists a linear map $\lambda:H^2(H,\Q_\ell)\to
  \Q_\ell$ such that  $\lambda(q(\alpha\cup \beta))$ defines a symplectic
  pairing  on $H^1(H,V)$.

\end{thm}

\begin{proof}
Suppose that $G=\pi_1^{{p'}}(X)$, where $X$ is an $n$ dimensional
smooth projective variety. Fix an ample line bundle $\OO_X(1)$, and
let $L$ denote the corresponding Lefschetz operator.
We claim that
\begin{equation}
  \label{eq:HL}
H^1(X,V)\times H^1(X,V) \stackrel{q\circ \cup}{\longrightarrow} H^2
(X,\Q_\ell)\stackrel{L^{n-1}}{\longrightarrow} H^{2n}(X,\Q_\ell)\cong
\Q_\ell
\end{equation}
gives a nondegenerate symplectic pairing. When $\rho$ is trivial,
$(V,q)$ is a sum of $n$ copies of $\Q_\ell$ with the standard pairing, and
the claim follows from  the hard Lefschetz theorem for  \'etale cohomology
\cite[4.1.1]{deligne}. In general, let $\pi:Y\to X$ be a Galois \'etale
cover with Galois group $K$ such that $\pi^*V$ is
trivial. We can decompose $\pi_*\pi^*V=V\oplus V'$ under the $K$
action. Let $p:\pi_*\pi^* V\to V$ denote the projection.
We equip $Y$ with the Lefschetz operator corresponding to
$\pi^*\OO_X(1)$. Then  the pairing \eqref{eq:HL}
is obtained by applying $p$ to the  nondegenerate pairing
$$
H^1(Y,\pi^*V)\times H^1(Y,\pi^*V) {\longrightarrow} H^2
(Y,\Q_\ell)\stackrel{L^{n-1}}{\longrightarrow} H^{2n}(Y,\Q_\ell)\cong
\Q_\ell
$$
and the claim follows for general $V$ with finite monodromy.

By proposition~\ref{prop:cup}, we have a commutative diagram
$$
\xymatrix{
 H^1(H,V)\times H^1(H,V)\ar[r]\ar[d]^{\cong} & H^2(H,\Q_\ell)\ar[d]^{\iota}\ar[rd]^{\lambda} &  \\ 
 H^1(X,V)\times H^1(X,V)\ar[r]& H^2(X,\Q_\ell)\ar[r]^{L^{n-1}} & \Q_\ell
}
$$
where $\lambda$ is defined as $L^{n-1}\circ \iota$.
\end{proof}

\begin{proof}[Second proof of corollary \ref{cor:even}]
  $H^1(G, \Q_\ell)$ carries a symplectic pairing, so it must be even dimensional.
\end{proof}

The theorem itself gives more subtle information than the parity
test. For example, we have the following consequence.

\begin{prop}
 Suppose that
$$1\to K\to G\to H\to 1$$
is an extension of pro-$p'$ groups such that   $H^1(H,\Q_\ell)\not=0$
and the trangression
$Hom_H(K,\Q_\ell)\to H^2(H,\Q_\ell)$ is an isomorphism. Then $G\notin \cP(p)$
\end{prop}

\begin{proof}
  From the Hochschild-Serre  sequence (lemma \ref{lemma:HS})
$$
 0\to  H^1(H,\Q_\ell)\stackrel{\alpha}{\longrightarrow} H^1(G,
\Q_\ell)\to  Hom_H(K,\Q_\ell)\stackrel{\sim}{\longrightarrow}
H^2(H, \Q_\ell)\stackrel{\beta}{\longrightarrow} H^2(G,\Q_\ell) 
$$
we conclude that $\alpha$ is an isomorphism and $\beta=0$.
Therefore  
$$\cup: \wedge^2 H^1(G,\Q_\ell)\to H^2(G,\Q_\ell)$$
is zero, because it factors through $\beta$. Thus $G\notin \cP(p)$ by theorem~\ref{thm:even}.
\end{proof}

The conditions of the proposition are easy to check for the following
example.

\begin{cor}
The completion of the Heisenberg group
$$\left \{
\begin{pmatrix}
1 & a &b\\0 & 1& c\\ 0&0&1  
\end{pmatrix}\mid a,b,c\in \hat\Z\right\}$$
is not in $\cP(p)$
\end{cor}

A more general class of examples where the proposition
applies comes from generalized universal central extensions.
Given a pro-$p'$ group $H$, we have a (generally  noncanonical) 
central extension
\begin{equation}
  \label{eq:univ}
  0\to H_2(H, \hat \Z)\to G\to H\to 1
\end{equation}
with extension class lifting  the identity under the surjection
$$H^2(H, H_2(H,\hat\Z))\to Hom(H_2(H,\hat\Z),H_2(H,\hat\Z))$$
Here the (co)homologies are defined by taking inverse limits of the
usual groups with coefficients in $\Z/n\Z$. Transgression gives an
isomorphism, so that:

\begin{cor} 
If $H^1(H,\Q_\ell)$ is  nonzero, then the group $G$ of the above
extension \eqref{eq:univ} is not in $\cP(p)$.
\end{cor}

The last statement should be compared with \cite[p 717, cor]{reznikov}.

\section{Free products}

 Given two pro-$p'$ groups $G_1$ and $G_2$,  their coproduct in the
 category of pro-$p'$ groups exists \cite[\S 9.1]{rz}. We denote it by
  $G_1\hat * G_2$.  It is closely  related
  to the   usual free product $*$.

 \begin{lemma}\label{lemma:freeprod}
   Given discrete groups $G_i$,
    $\widehat{G_1 * G_2} \cong \hat G_{1}\hat* \hat G_{2}$.
  \end{lemma}

  \begin{proof}
\cite[9.1.1]{rz}.    
  \end{proof}

The completion $\hat F^r$ of the usual free group on $r$
  generators is a free pro-$p'$ group. It can also be expressed as
a coproduct
  $$\hat F^r=\hat \Z\hat *\ldots \hat *\hat \Z \quad (\text{$r$ factors})$$

In \cite{abr}, it is shown that a K\"ahler group cannot be an
extension of  a group with infinitely many ends by a finitely
generated group. We observe that any nontrivial free product other than
$(\Z/2\Z)*(\Z/2\Z)$ has infinitely many ends.
Since we do not (yet) have a theory of ends in the
profinite setting, we give a slightly weaker statement involving the
aforementioned  class. On the other hand, the hypothesis on  the
kernel can be relaxed slightly.

\begin{thm}\label{thm:coprod}
Let $p\not=2$.  Suppose that we have an extension of pro-$p'$ groups
 \begin{equation}
   \label{eq:KGH}
   1\to K\to G\to H\to 1
 \end{equation}
such that 
\begin{enumerate}
\item[(a)] $(K/[K,K])_2/(\text{torsion})$ is a finitely generated
  $\Z_2$-module, and
\item[(b)] $H$ is a nontrivial coproduct other than $(\Z/2\Z)\hat*(\Z/2\Z)$.
\end{enumerate}
Then $G\notin \cP(p)$
\end{thm}

\begin{proof} Suppose that  $G$ fits into the exact sequence
  \eqref{eq:KGH}  with $H=H_1\hat*H_2$, where $H_i$ are nontrivial and
  not both of order $2$.   We will show that  $G$ cannot lie
in  $\cP(p)$. We first make a reduction to the case where $H$  is of the form
  $J\hat* \hat F^2$.
Choose nontrivial finite quotients $Q_i$ of
$H_i$ such that $|Q_i|>2$ for some $i$.
Let $L\subset H$ be the kernel of the projection $H\to Q_1\times Q_2$.
 Then by  the profinite version of the Kurosh subgroup theorem \cite[9.1.9]{rz},
we see that $L\cong J\hat* \hat F^2$ for some group $J$.  It suffices to prove that the
preimage $\tilde G$ of $L$ in $G$ is not  in $\cP(p)$ by lemma~\ref{lemma:fi}.
Since it  fits into an extension
$$1\to K\to \tilde G\to L\to 1$$
we may replace $G$ by $\tilde G$ and $H$ by $L$.

From the exact sequence
\eqref{eq:KGH}, we get a continuous action of $H$  on $K/[K,K]$.
 Therefore 
$M=(K/[K,K])_2 \otimes_{\Z_2} \Q_2$  is a finite dimensional
representation of $H$.  With respect to the factor 
$\hat F^2=\hat \Z\hat* \hat \Z$ of $H$, we get two actions of $\hat \Z$ on $M$ that we refer
to as the first and second.  Let $\{\xi_1,\ldots, \xi_n\}$ be the
(possibly empty) set of one dimensional characters of $\hat F^2$ corresponding
to one dimensional subquotients of $M$. We may suppose that $\xi_1,\ldots, \xi_m$ are the
characters among these with finite order. Let  $S\subset\hat  F^2$ be the
intersection of kernels of $\xi_1,\ldots, \xi_m$.
The group $S$ is necessarily of the form $\hat F^r$ with
$r\ge 2$ \cite[3.6.2]{rz}. After replacing $H$ by $J \hat * S= (J\hat*
\hat F^{r-2})\hat * \hat F^2$, $J$ by $J\hat*
\hat F^{r-2} $ and $G$ by the  preimage of the new $H$ in the old $G$, we may assume that  the all the characters $\xi_i$
are either trivial or of infinite order.  Let $\xi_i'$ denote the
restrictions of $\xi_i$ to the first factor of $\hat F^2=\hat \Z\hat
*\hat \Z$. Then the 
sign character $\sigma:\hat \Z\to \Z_2\to \Q_2^*$,
$$
\sigma(x)=
\begin{cases}
  +1 &\text{if $x\in 2\hat \Z$}\\
 -1 & \text{otherwise}
\end{cases}
$$
is not in $\{\xi_1',\ldots, \xi_n'\}$.
Let $\chi_1=\sigma$ and $\chi_2\in \{1,\sigma\}$, where the precise choice
will be determined  below.
 Let $V=\Q_2$
denote the $H=J\hat* \hat \Z\hat* \hat \Z$ module where $J$ acts trivially
and the two $\hat \Z$ factors act through $\chi_1$ and $\chi_2$
respectively. We note that $V$ is orthogonal, so that we can apply
theorem~\ref{thm:even} when the time comes.

We now
compute $\dim H^1(G, V)$. From the Hochschild-Serre 
sequence (lemma \ref{lemma:HS}), we obtain the exact sequence
\begin{equation}
  \label{eq:HS}
0\to H^1(H, V)\to H^1(G, V) \to H^0(H, H^1(K, V))  
\end{equation}
We can identify
$$H^0(H, H^1(K, V))\cong H^0(H, Hom(K/[K,K],  V))\cong
Hom_{H}(M, V)$$
Since we chose $\chi_1\notin \{\xi_1',\ldots, \xi_n'\}$,  the latter
space is zero. Therefore, by \eqref{eq:HS} we obtain an isomorphism
$$H^1(G,V) \cong H^1(H,V)$$
By an appropriate
 Mayer-Vietoris sequence \cite[prop 9.2.13]{rz}, we see that
$$H^1(H, V) \cong H^1(J,V)\oplus H^1(\hat \Z, \Q_{2,
  \chi_1})\oplus H^1(\hat \Z,\Q_{2, \chi_2})$$
where the subscripts $\chi_i$ indicate the action.
The middle group on the right  vanishes because $\chi_1$ was nontrivial.
By choosing $\chi_2$  to be trivial or not according to the parity of
$rank(J/[J,J])$, we see that the
right side can be made to have odd dimension. Therefore
$G$ cannot be the pro-$p'$ fundamental group of a smooth projective
variety by theorem~\ref{thm:even}.
\end{proof}

\begin{cor}\label{cor:free}
  A group in $\cP(p)$  cannot decompose as a coproduct of nontrivial 
 pro-$p'$ groups, and in particular it cannot be free.
\end{cor}

\begin{proof}
  The only case not covered by the last theorem is
  $(\Z/2\Z)\hat*(\Z/2\Z)$, but since this contains $\hat \Z$ as an open
  subgroup, it is ruled out by corollary~\ref{cor:even}.
\end{proof}

\begin{cor}
  Suppose that $G$ satisfies all of the assumptions of the theorem but
  with {\rm (a)} replaced by 
  \begin{enumerate}
  \item[(a')] $K$ is topologically finitely generated. 
  \end{enumerate}
Then $G\notin \cP(p)$.
\end{cor}

\begin{proof}
  (a') implies (a).
\end{proof}

\begin{cor}\label{cor:hatABR}
  Suppose that  $1\to K\to G\to H_1*H_2\to 1$ is an exact sequence of discrete
  groups, with $K$ finitely generated and $\hat H_{i}$ nontrivial
  and not both of order $2$. Then $\hat G\notin \cP(p)$.
\end{cor}

\begin{proof}
  By  \cite[prop 3.2.5]{rz} and lemma~\ref{lemma:freeprod}, we have an exact sequence
$$ \hat K\to \hat G\stackrel{f}\to \hat H_{1}\hat * \hat
H_{2}\to 1$$
Therefore $\ker f$ is topologically finitely generated.
\end{proof}

As an illustration of the use of this theorem, we show that the pure
braid group does not lie in this class. This is a direct translation of the argument in \cite{arapura} for
showing that  braid groups are not K\"ahler.  Recall that $B_n$ is
given by generators $s_1,\ldots, s_{n-1}$ with relations $s_is_{i+1}s_i=
s_{i+1}s_is_{i+1}$ and $s_is_j=s_js_i$ if $|i-j|>1$. This maps to the
symmetric group $S_n$by sending $s_i\mapsto (i\, i+1)$. The kernel is
the pure braid group $P_n$. More geometrically, $P_n$ is the
fundamental group of the configuration space of $n$ distinct ordered points in
the plane.

\begin{prop}
$\hat P_{n}\notin \cP(p)$.
\end{prop}

\begin{proof}
We have $P_2=\Z$, so $\hat P_{2}\notin \cP(p)$ by
corollary~\ref{cor:even}. The group $B_3$ is generated by $a=s_1s_2s_1$
and $b=s_1s_2$ with the relation $a^2=b^3$.
There is a surjective homomorphism from $f:B_3\to \Z/2\Z *
  \Z/3\Z$ which sends $a$ and $b$ to the generators of $\Z/2\Z $ and $\Z/3\Z$
respectively. The kernel of $f$  is the cyclic group generated by
$a^2\in P_3$.  Thus we have an extension
$$0\to \Z\to P_3\to f(P_3)\to 1$$
By Kurosh's subgroup theorem \cite[\S 5.5]{serre}  the  image $f(P_3)$ is a free
product of a nonabelian free group and some additional factors. Therefore $\hat P_{3}\notin \cP(p)$ by
corollary~\ref{cor:hatABR}. When $n>3$, projection of the configuration
spaces gives a fibration resulting a surjective homomorphism $P_n\to P_3$
with finitely generated kernel. It follows that $P_n\to f(P_3)$ is
again surjective with finitely generated kernel. So once again corollary~\ref{cor:hatABR}
shows that $\hat P_{n}\notin \cP(p)$.

\end{proof}

\section{One relator groups}

Recently, Biswas-Mahan \cite{bm} and Kotschick \cite{kotschick} have classified one relator
K\"ahler groups:  they are all fundamental groups of one dimensional
compact  orbifolds with at most one orbifold point. In more explicit terms,
such a group  would be of the form
$$\Gamma_{g,m}=  \begin{cases}
  \langle x_1,\ldots, x_{2g}\mid ([x_1,x_{g+1}]\ldots
[x_g,x_{2g}])^m\rangle & \text{if $g>0$}\\
\Z/m\Z= \langle x\mid x^m\rangle & \text{if $g=0$}
\end{cases}
$$
(Note that both \cite{bm, kotschick} classify infinite one relator
K\"ahler groups, but the statement as given above is an immediate consequence.)
 We prove a  pro-$\ell$ version for large $\ell$ assuming that the
 relation lies in the commutator subgroup. To reconcile  the statement
 below with the one just given,
 observe that $(\hat \Gamma_{g,m})_\ell\cong (\hat \Gamma_{g,1})_\ell$ when  $\ell$ is coprime to $m$.

\begin{thm}\label{thm:1rel}
  Suppose that $G\in \cP(p)$ is the pro-$p'$  completion of a
  discrete one-relator group.  Then there exists an explicit finite  set $S$ of
  primes such that if
  $\ell\notin S$,   the maximal pro-$\ell$ quotient $G_\ell$ of  $G$ is isomorphic to the pro-$\ell$ completion of the genus
  $g$ surface group $\Gamma_{g,1}$
where $g=\frac{1}{2}\dim H^1(G,\Q_\ell)$.
\end{thm}

Before giving the proof, we need the following version of Stallings'
theorem \cite{stallings}.

\begin{lemma}\label{lemma:stallings}
If $f:G\to H$ is a continuous homomorphism of pro-$\ell$ groups such
that the induced map  $H^i(H,\Z/\ell\Z)\cong H^i(G,\Z/\ell\Z)$ is an
isomorphism for $i=1$ and an injection for
$i=2$, then $f$ is an isomorphism. 
\end{lemma}

\begin{proof}
The surjectivity of $f$ follows from \cite[I prop 23]{serreG}, so it
remains to  check injectivity. 
  Define the $\ell$-central series by $C^0(G)=G$, and $C^{n+1}(G)=
  [G,C^n(G)]C^n(G)^\ell$. We
  claim that $f$ induces an isomorphism $G/C^n(G)\to H/C^n(H)$. The
 injectivity of $f$ will follow from this claim because one has $\bigcap C^n(G)=\{1\}$
The  proof of  the claim is essentially identical to the  argument in
\cite{stallings} in dual form; nevertheless we give it for completeness. This proof goes by induction. The initial case $n=1$
follows from the isomorphism $H^1(H,\Z/\ell\Z)\cong
H^1(G,\Z/\ell\Z)$. For the induction step we use the commutative diagram,
$$
\xymatrix{
 1\ar[r] & C^nH/C^{n+1}H\ar[r]\ar[d]^{\gamma} & H/C^{n+1}H\ar[r]\ar[d]^{f_{n+1}} & H/C^nH\ar[r]\ar[d]^{f_n} & 1 \\ 
 1\ar[r] & C^nG/C^{n+1}G\ar[r] & G/C^{n+1}G\ar[r] & G/C^nG\ar[r] & 1
}
$$
We have to show that $f_{n+1}$ is an isomorphism assuming this for
$f_n$. It is enough to check that $\gamma$ is an isomorphism.
We have a diagram coming from Hochschild-Serre,
$$
\xymatrix{
 H^1(H/C^n)\ar[r]\ar[d]^{\alpha} & H^1(H)\ar[r]\ar[d]^{\beta} & Hom(C^nH,\Z/\ell)\ar[r]\ar[d]^{\gamma^*} & H^2(H/C^n)\ar[r]\ar[d]^{\delta} & H^2(H)\ar[d]^{\epsilon} \\ 
 H^1(G/C^n)\ar[r] & H^1(G)\ar[r] & Hom(C^nG,\Z/\ell)\ar[r] & H^2(G/C^n)\ar[r] & H^2(G)
}
$$
The hypotheses, including the induction hypothesis, implies that
$\alpha,\beta,\delta$ are isomorphisms, and $\epsilon$ is
injective. Therefore $\gamma^*$ is an isomorphism by the 5-lemma.
This implies that $\gamma$ is an isomorphism.

\end{proof}

\begin{proof}[Proof of theorem \ref{thm:1rel}]
Let $G$ be the completion of the quotient of the free group on $d$
letters $F=F^d$ by the normal
subgroup $R$ generated by single element $r\in F$ with  $r\not= 1$.

 We have two cases.  The first case is where $r\in [F,F]$.
The associated graded of $F$ with respect to the lower central series
$$Gr(F) = F/[F,F]\oplus [F,F]/[F,[F,F]]\oplus\ldots$$
is a graded Lie algebra over $\Z$ with Lie bracket induced by the
commutator \cite{lazard}. 
Let $x_1,\ldots, x_d$ denote
generators of $F$. The first summand $F/[F,F]$ is a  free $\Z$-module freely 
generated by the classes $\bar x_i$ of $x_i$, and the 
next summand $[F,F]/[F,[F,F]]$ is freely 
generated by $[\bar x_i,\bar x_j],$ with $ i<j$. Thus we can expand the class $\bar r\in
[F,F]/[F,[F,F]]$ of $r$ as $\bar r=\sum a_{ij}[\bar x_i, \bar
x_j]$ with $a_{ij}\in \Z$. We extend
$(a_{ij})$ to a skew symmetric matrix by setting $a_{ji}=-a_{ij}$ and
$a_{ii}=0$. By \cite[3.2.5]{rz}, we have an  exact sequence
$$ \hat R_\ell\to \hat F_\ell\to G_\ell\to 1$$
Thus  $\hat G_\ell$ is also a one relator group in the
topological sense, and   therefore
$\dim_{\mathbb{F}_\ell}H^2(G_\ell,\Z/\ell\Z)=1$ by  \cite[cor, p 31]{serreG}.
 We can also conclude that 
$$G_\ell/[G_\ell,G_\ell]\cong \hat F_\ell/[\hat F_\ell,\hat F_\ell]\cong 
\Z_\ell^d$$
Thus $d=\dim_{\mathbb{F}_\ell}H^1(G_\ell,\Z/\ell\Z)$ is the minimal number of
generators of $G_\ell$. 
From  theorem~\ref{thm:even}, it follows that  $d=2g$  for some integer $g$ and $H^2(G_\ell,\Q_\ell)\not=0$.
Therefore $H^2(G_\ell,\Z_\ell)\cong \Z_\ell$ and the cup product
pairing
$$H^1(G_\ell,\Z_\ell)\times H^1(G_\ell,\Z_\ell)\to H^2(G_\ell,\Z_\ell)$$
is nondegenerate.  By an argument identical to the proof of \cite[prop
3]{labute}, we see that this pairing is represented by  the matrix
$(a_{ij})$. Let
$S$ denote the union of $\{2,p\}$ and  the set of all prime factors of
the $a_{ij}$. Then 
we can reduce modulo $\ell\notin S$ to obtain a nondegenerate cup
product pairing
$$H^1(G_\ell,\Z/\ell\Z)\times H^1(G_\ell,\Z/\ell\Z)\to H^2(G_\ell,\Z/\ell\Z)$$
 It follows that $G_\ell$ is a so called  Demushkin group
\cite{demushkin, labute}. These are classified. Since $\ell$ is odd,
the only possibility is 
$$G_\ell \cong  \langle y_1,\ldots, y_{2g}\mid  y_1^{\ell^n}[y_1,y_{g+1}]\ldots
[y_g,y_{2g}]\rangle$$
for some integer $n\ge 0$.
When $n>0$, $G_\ell/[G_\ell,G_\ell]$ has torsion contrary to what was
shown above. Therefore $n=0$ and the theorem is proved in this case.

Now we turn to the remaining case where $r\notin [F,F]$. Let $\bar r\in F/[F,F]\cong \Z^d$ be the image
of $r$.  Fix an isomorphism
$$( F/[F,F])/(\bar r)/(\text{torsion})\cong \Z^{d-1}$$
and lift the generators on the right to the free group $F'=F^{d-1}$.
We thus have a commutative diagram
$$
\xymatrix{
 F\ar[r]\ar@{-->}[d]^{\phi}\ar[rd] &  \Z^d\ar[d]\ar[rd]&  \\ 
 F'\ar[r] & \Z^{d-1} & \Z^d/(\bar r)\ar[l]
}
$$
given by the solid arrows. We can choose a  homomorphism $\phi:F\to F'$
which completes the commutative diagram as indicated.
Let $K$ be the quotient of $F'$ by the  normal subgroup generated by
$r'=\phi(r)$, and let $H=  \hat K$. The homomorphism $\phi$
induces a continuous homomorphism $f:G\to H$.

Let $S_1$ be the set of primes $\ell$ such that
$\Z^d/(\bar r)$ has $\ell$-torsion. Equivalently $S_1$ is the minimal
set of primes such that $(\Z^d/(\bar r))_\ell$ is
torsion free whenever $\ell\notin S_1$. We assume that $\ell\notin
S_1$ for the remainder of this paragraph.
We claim that $f$ induces an isomorphism $G_\ell\cong H_\ell$. By construction, we have $H^1(H_\ell,
\Z/\ell\Z)\cong H^1(G_\ell,\Z/\ell\Z)$. If we can show that there is
an injection on $H^2$, the claim will follow from lemma~\ref{lemma:stallings}.
We split this into subcases. Suppose that $r'=1$, then
 $H_\ell$ is a free pro-$\ell$ group. Therefore $H^2(H_\ell)=0$ so the
claim follows in this case. But in fact, this case is impossible
because $G_\ell$ cannot be free. Thus $r'\not=1$. Then
$$H^2(H_\ell,\Z/\ell\Z)\cong H^1(R',\Z/\ell \Z)^{\hat F_\ell'}\cong
\Z/\ell\Z $$
where $R'\subset \hat F_\ell'$ is the closed normal subgroup generated by $r'$
(cf \cite[pp 30-31]{serreG}). We have a similar description for
$H^2(G_\ell)$. From this it follows easily that the map $H^2(H_\ell)\to
H^2(G_\ell)$ is nonzero, and therefore an isomorphism.

With the claim now proven, we can work with $H$ instead of $G$ provided we choose
$\ell\notin S_1$. By construction, $r'\in [F',F']$, so  we are now
in the same situation as case 1. 
The  arguments for that case show there is a finite set of primes $S_2$,
explicitly determined by $r'$, such that when $\ell\notin S_2$, we have $H_\ell\cong \hat
\Gamma_{g,1;\ell}$ for some $g$. In conclusion, the theorem holds in
the second case when $S=S_1\cup S_2$.

\end{proof}

\begin{cor} With the notation as in  theorem \ref{thm:1rel}, the maximal pro-nilpotent  prime to $S$ quotients of $G$ and $\hat \Gamma_{g,1}$ are isomorphic.
\end{cor}

\begin{proof}
This follows from the theorem and \cite[lemma 2.10]{lo}.
\end{proof}

\end{document}